\documentclass[12pt, oneside,reqno]{amsart}
\usepackage{amsmath, amsfonts, amssymb}
\usepackage{eucal}
\usepackage{mathrsfs}
\usepackage{latexsym}
\usepackage{cite}
\usepackage{bbm}
\usepackage{todonotes}
\usepackage{mathtools}
\usepackage{graphicx}
\usepackage{relsize}
\usepackage{exscale}
\newcommand{\proofbox}{\mbox{ $\Box$}\\}
\newcommand{\R}{\mathbb{R}}

\newcommand{\N}{\mathbb{N}}

\newcommand{\HH}{\mathcal{H}}

\newtheorem{lemma}{LEMMA}[section]
\newtheorem{theo}[lemma]{THEOREM}

\newtheorem{coro}[lemma]{COROLLARY}

\newtheorem{prop}[lemma]{PROPOSITION}

\newtheorem{problem}[lemma]{PROBLEM}

\title[The Weighted Dual Minkowski Problem Under Group Symmetry]
{The Weighted Dual Minkowski Problem Under Group Symmetry}

\author[K. B\"or\"oczky]{K\'aroly J. B\"or\"oczky}
\address{Alfr\'ed R\'enyi Institute of Mathematics,
 Hungarian Academy of Sciences,
 Realtanoda u. 13-15, H-1053, and ELTE, Mathematical Institute,
 Budapest, Hungary}
\email{boroczky.karoly.j@renyi.hu}

\author[S. Mui]{Stephanie Mui}
\address{Department of Mathematics, Georgia Institute of Technology,  686 Cherry St NW, Atlanta, GA 30332, USA}
\email{smui3@gatech.edu}

\author[G. Zhang]{Gaoyong Zhang}
\address{Department of Mathematics,
Courant Institute of Mathematical Sciences,
New York University,
251 Mercer Street,
New York, NY 10012, USA}
\email{gaoyong.zhang@nyu.edu}


\begin{document}

\maketitle

\begin{abstract}
  The paper studies the weighted dual Minkowski problem for the weighted $q$th dual curvature measure
  of convex bodies in $\mathbb R^n$ first posed by Huang, Lutwak, Yang, and Zhang.
  Most of the results are in the $G$-invariant setting,
  i.e. when the convex bodies are invariant under a closed subgroup $G$ of $O(n)$.
  They generalize previous results on the existence of solutions to the dual Minkowski problem
  for origin-symmetric convex bodies.
 It is proved:
 (i) a full characterization of the  $G$-invariant weighted dual curvature measure for $q\in(0,1]$,
 (ii) a necessary and sufficient condition for the existence of solutions to the $G$-invariant weighted
 dual Minkowski problem for $q\in(1,n)$,
 (iii) a sufficient condition for the $G$-invariant weighted dual curvature measure for $q>1$, and
 (iv) an extension of Henk-Pollehn's dual curvature measure concentration property.
\end{abstract}

\section{Introduction}

The classical Minkowski problem is a central question in the Brunn-Minkowski theory of convex bodies
and is concerned with the characterization of the surface area measure. Namely, it asks what kind of
Borel measures on the unit sphere can be the surface area measure of a convex body, which can be defined
in a direct way as follows. Let $\partial'K$ denote the subset of the boundary  of a convex body
$K\subset\R^n$ where there is a unique outer unit normal vector.
It is well-known that $\partial K\backslash \partial'K$ is the countable union of compact sets of finite
$\HH^{n-2}$-measure (see Schneider \cite[{Theorem~2.2.5}]{Sch14}), and hence $\partial'K$ is Borel
and $\HH^{n-1}(\partial K\setminus\partial'K)=0$. Then $\nu_K:\partial'K\to S^{n-1}$ is a function
that is known as the spherical Gauss map, and
 $\nu_K$ is continuous on $\partial'K$.
 The surface area measure of $K$, denoted by $S(K,\cdot)$, is a Borel measure on $S^{n-1}$ such that
 for any Borel set $\eta\subset S^{n-1}$, we have
$S(K,\eta)=\HH^{n-1}({\nu}_K^{-1}(\eta))$. An important property of the surface area measure is that
it satisfies Alexandrov's variational formula. Namely,  if $\varphi:S^{n-1}\to\R$ is continuous
and $K_t$ is the corresponding Wulff-shape
$$
K_t=\{x\in\R^n:\,\langle x,u\rangle\leq h_K(u)+t\varphi(u)\,\;\forall u\in S^{n-1}\}
$$
when $|t|$ is small, then
\begin{equation}
\label{Alexandrov}
\lim_{t\to 0}\frac{V(K_t)-V(K)}{t}=\int_{S^{n-1}}\varphi(u)\,d S(K, u).
\end{equation}

The classical Minkowski problem asks for necessary and sufficient conditions for a Borel measure
on $S^{n-1}$ to be the surface area measure of a convex body.
Similar questions have been posed, and at least partially solved, for other measures associated with
convex bodies in the Brunn-Minkowski theory. Some important examples include the integral curvature
measure $J(K,\cdot)$ of Alexandrov (see (\ref{Jdef}) below), or the $L_p$ surface area measure
$dS_p(K,\cdot)=h_K^{1-p}dS(K,\cdot)$ for $p\in \R$ introduced by Lutwak \cite{Lut93a}, where
$S_1(K,\cdot)=S(K,\cdot)$ ($p=1$) is the classical surface area measure, and $S_0(K,\cdot)$ ($p=0$) is the
 cone volume measure. Here, some care is needed if $p>1$, when we only consider the case $o\in\partial K$
 if the resulting $L_p$ surface area measure $S_p(K,\cdot)$ is finite.
For a detailed overview of these measures and their associated Minkowski problems, see Schneider \cite{Sch14},
and Huang-Lutwak-Yang-Zhang \cite{HLYZ16}.

Lutwak built the dual Brunn-Minkowski theory in the 1970s as a  ``dual" counterpart of the classical theory.
Although there is no formal duality between the classical and dual theories, one can say roughly that
in the dual theory, the radial function plays a similar role to the support function in the classical theory.
The dual Brunn-Minkowski theory concerns the class $\mathcal{S}^n_{o}$ of measurable star shaped sets of $\R^n$,
where $Q\subset\R^n$ is called star shaped if
\begin{itemize}
\item $\exists R>r>0$ such that $rB^n\subset Q\subset RB^n$;
\item $\lambda x\in Q$ for any $x\in Q$ and $\lambda\in[0,1]$.
\end{itemize}
 Clearly,  $\mathcal{K}^n_{o}\subset \mathcal{S}^n_{o}$. For a star shaped set $Q\in\mathcal{S}_o^n$,
 we define the radial function and the norm corresponding to $Q$ as
\begin{align*}
\varrho_Q(u)&=\max\{t\geq 0:\,tu\in Q\} \mbox{ \ \ for }u\in S^{n-1}\\
\|x\|_Q&=\min\{t\geq 0:\,x\in tQ\}\mbox{ \ \ for }x\in\R^n,
\end{align*}
and hence $\varrho_Q(u)=1/\|u\|_Q$ for $u\in S^{n-1}$.
It follows that radial functions of star shaped sets are the bounded positive measurable functions on $S^{n-1}$.

Dual intrinsic volumes $\widetilde{V}_q(K)$ for convex bodies $K\in\mathcal{K}^n_{(o)}$ for $q\in\R$ were
defined by Lutwak \cite{Lut75}, and Lutwak-Yang-Zhang \cite{LYZ18} extended the definition to the weighted
case with respect to
a $Q\in\mathcal{S}_o^n$ as
\begin{align}
\label{dualintrvol}
\widetilde{V}_q(K)&=\frac1n\int_{S^{n-1}}\varrho^q_K(u)\,d\mathcal{H}^{n-1}(u)=
\frac{q}n\int_K\|x\|^{q-n}\,dx\\
\label{dualintrvolQ}
\widetilde{V}_q(K,Q)&=\frac1n\int_{S^{n-1}}\varrho^q_K(u)\varrho^{n-q}_Q(u)\,d\mathcal{H}^{n-1}(u)
=\frac{q}n\int_K\|x\|_Q^{q-n}\,dx,
\end{align}
where the second equality follows from using polar coordinates.
The $q$th dual intrinsic volume is $q$-homogeneous. Namely, if $\lambda>0$, then
$$
\widetilde{V}_q(\lambda\,K,Q)=\lambda^q\widetilde{V}_q(K,Q).
$$
According to Lutwak-Yang-Zhang \cite{LYZ18}, in the weighted case \eqref{dualintrvol}, one has that
for any $\Phi\in {\rm SL}(n)$,
$$
\widetilde{V}_q(\Phi\,K,\Phi\,Q)=\widetilde{V}_q(K,Q).
$$

We note that
the so-called cone volume measure defined as $V(K,\cdot)=\frac1n\,S_0(K,\cdot)=\frac1n\,h_KS(K,\cdot)$
and Alexandrov's integral curvature measure $J(K,\cdot)$ of a convex body $K\in\mathcal{K}_{o}^n$
can both be represented as dual curvature measures in the following way
\begin{eqnarray}
\label{conevol}
\mbox{$V(K,\cdot)=\frac1n$}\,S_0(K,\cdot)&=& \widetilde{C}_n(K,B^n;\cdot)\\
\label{Jdef}
J(K^*,\cdot)&=& \widetilde{C}_0(K,B^n,\cdot).
\end{eqnarray}
And hence, the dual curvature measures interpolate between well-known measures, including
the integral curvature measure and the cone-volume measure.

The papers of Huang-Lutwak-Yang-Zhang \cite{HLYZ16} and Lutwak-Yang-Zhang \cite{LYZ18} posed
the following fundamental question:

\begin{problem}[Dual Minkowski problem]
\label{LpdualMinkowskiProblem}
Let $q\in\R$ and $Q\in\mathcal{S}^n_{o}$.
Find the necessary and sufficient conditions such that for a given finite Borel measure $\mu$ on $S^{n-1}$,
there exists a convex body $K\in\mathcal{K}_{o}^n$ that is a solution to the dual curvature measure equation,
\[
\widetilde{C}_{q}(K,Q,\cdot)=\mu.
\]
In the continuous case,  the measure equation becomes the Monge-Amp\`{e}re equation on $S^{n-1}$,
%(see  \eqref{Monge-AmpereQ0} in Section~\ref{secregularity})
\begin{equation}
\label{Monge-AmpereQ}
\det(\nabla^2 h(u)+h(u)\,{\rm Id})=\mbox{$\frac1n$}\,\|\nabla h(u)+h(u)\,u\|_Q^{n-q}  f(u).
\end{equation}
\end{problem}

When $Q=B^n$, the Monge-Amp\`ere equation becomes
\begin{equation}
\label{MongeDualMink}
\det(\nabla^2 h(u)+h(u)\,{\rm Id})=\mbox{$\frac1n$}\,\|\nabla h(u)+h(u)\,u\|^{n-q} f(u),
\end{equation}
and the following results are known:
\begin{itemize}
\item If $q<0$, then any Borel measure on $S^{n-1}$
not concentrated on a closed hemisphere is a $q$th dual Minkowski curvature measure according to
Zhao \cite{Zha17} and Li-Sheng-Wang \cite{LSW20}.
\item If $q=0$, then Aleksandrov \cite{Ale42,Ale96} characterized $\widetilde{C}_{K,0}$, the so called
integral curvature measure (see also Oliker \cite{Oli07}, Bertrand \cite{Ber16},
and B\"or\"oczky-Lutwak-Yang-Zhang-Zhao \cite{BLYZ20}).
\item If $q>0$ and $n=2$, then  \eqref{MongeDualMink} has a solution for any measurable $f$,
provided $\frac1{c}<f<c$ for a $c>1$, according to Chen-Li \cite{ChL18}.
\item If $0<q<n$, then a finite even Borel measure $\mu$ on $S^{n-1}$ is a $q$th dual Minkowski
curvature measure if and only if
\begin{equation}
\label{subspace-concentration}
\mu(L\cap S^{n-1})<\mbox{$\frac{{\rm dim}\,L}q$}\cdot \mu(S^{n-1})
\end{equation}
for any proper linear subspace $L\subset \R^n$, according to B\"or\"oczky-Lutwak-Yang-Zhang-Zhao
\cite{BLYZ19},
where one must add that $\mu$ is not concentrated onto a great subsphere if $q<1$.
These results have been partially extended to centered convex bodies by
Eller-Henk \cite{ElH23}.
\item If $q=n$, then any measure satisfying \eqref{subspace-concentration} is a cone volume measure
($\widetilde{C}_{K,n}$)
according to Chen-Li-Zhu \cite{CLZ19}. There are obstructions for a finite Borel measure to be
a cone volume measure, according to B\"or\"oczky-Heged\H{u}s \cite{BoH15}, and a variant of
\eqref{subspace-concentration} characterizes even cone volume measures according to
B\"or\"oczky-Lutwak-Yang-Zhang \cite{BLYZ13}.
\item If $q\geq n+1$ and $K\in\mathcal{K}^n_{(o)}$ is origin symmetric, then Henk-Pollehn
\cite{HeP18} prove the necessary condition
$$
\widetilde{C}_{K,q}(L\cap S^{n-1})<\mbox{$\frac{q-n+{\rm dim}\,L}q$}\cdot \widetilde{C}_{K,q}(S^{n-1})
$$
for any proper linear subspace $L\subset \R^n$.
\end{itemize}
\noindent{\bf Remark.}
No uniqueness of the solution of the dual Minkowski problem \eqref{MongeDualMink} holds in general.
For example, in the case of $q>2n$ and $f\equiv 1$, even if one assumes that the solution $h$ is even
(see Chen-Chen-Li \cite{CCL21}).

Uniqueness of the solution of the $q$th $L_p$ dual Minkowski problem \eqref{Monge-AmpereQ} was
investigated by Chen-Chen-Li \cite{CCL21} and Li-Liu-Lu \cite{LLL22} when $Q=B^n$. The case of
$n=2$ and $f$ being a constant function in \eqref{Monge-AmpereQ} has been completely resolved by
Li-Wan \cite{LiW}. Uniqueness in the the isotropic case, when the $f$ in \eqref{Monge-AmpereQ}
is a constant, was verified for certain ranges of values of $p$ and $q$ by Ivaki-Milman \cite{IvM23}
in the even case and
by Hu-Ivaki \cite{HuI24} in the general case.

Orlicz versions of these Monge-Amp\`ere equations have been
considered by  Li-Sheng-Ye-Yi \cite{LSYY22}, Feng-Hu-Liu \cite{FHL22}
and Hu-Liu-Ma \cite{HLM23}
in the case of the Aleksandrov problem, by Xing-Ye \cite{XiZ20} in the case
of the dual Minkowski problem, and
Gardner-Hug-Weil-Xing-Ye \cite{GHWXY19,GHXY20},
Xing-Ye-Zhu \cite{XYZ22}, and
Liu-Lu \cite{LiL20} in the case
of the general $L_p$ dual Minkowski problem.

For the $G$-invariant case, the weighted $L_p$ dual Minkowski problem for negative $p$ was studied by
B\"or\"czky-Kov\'acs-Mui-Zhang \cite{BKMZ26}.

Another important related variant of the dual Minkowski problem is the so-called
``Chord Minkowski Problem"
({\it cf.} Lutwak-Xi-Yang-Zhang \cite{LXYZ}) and its $L_p$ version by  Xi-Yang-Zhang-Zhao
\cite{XYZZ23} for $p>0$ and
by Li \cite{YLia,YLib} for $p<0$. See also  Guo-Xi-Zhao \cite{GXZ} and
Xi-Yang-Zhang-Zhao \cite{XYZZ23}.
In addition, the
 ``Affine dual Minkowski problem" is proposed by Cai-Leng-Wu-Xi \cite{CLWX}.

Most results mentioned above when $q>0$ are about origin-symmetric convex bodies. The main aim
of this paper is to replace the origin-symmetry with the weaker $G$-symmetry in the study of
the weighted dual Minkowski problem. 
Under $G$-symmetry, 
Guan, Guan \cite{GuG02} considered a variant of the Christoffel-Minkowski problem,  B\"or\"oczky and Kalantzopoulos \cite{BoK22} solved the logarithmic Minkowski problem
and B\"or\"oczky and De \cite{BoD21} further considered the stability. 
B\"or\"oczky, Kov\'acs, Mui, and Zhang \cite{BKMZ26} studied the weighted $L_p$ dual Minkowski problem 
for functions under $G$-symmetry.
Shan \cite{Sha26+} studied the dual Minkowski problem for measures under $G$-symmetry. 
This paper investigates the weighted dual Minkowski problem for measures under $G$-symmetry.  
Our first result is a full characterization for the $G$-invariant weighted dual curvature measure if $0<q\leq 1$.

\begin{theo}
\label{q01-G-weighted}
Let $n\geq 2$, $q\in(0,1]$, $G\subset O(n)$ a closed subgroup without a non-zero fixed point,
and $Q\in\mathcal{S}^n_{o}$ be invariant under $G$.
A non-trivial, $G$-invariant, finite, Borel
measure $\mu$  on $S^{n-1}$ is the weighted dual curvature measure $\widetilde{C}_{q}(K,Q;\cdot)$
of a $G$-invariant
convex body $K$ in $\R^n$ if and only if $\mu$ is not concentrated on any great subsphere.
\end{theo}

\noindent
{\bf Remark.}

\begin{itemize}
\item Since $K$ is  $G$-invariant, it is centered; namely, its centroid is at the origin.

\item A typical suitable subgroup  $G\subset O(n)$ in Theorem~\ref{q01-G-weighted} is $G=\{{\rm Id},-{\rm Id}\}$.
In this case, $K$ and $\mu$ are $G$-invariant if and only if $K$ is $o$-symmetric and $\mu$ is even.

\item There are many classes of group invariant convex bodies that are not origin-symmetric,
see \cite{BKMZ26} for explicit constructions.
\end{itemize}

\begin{coro}
\label{q01-even-weighted}
Let $n\geq 2$, $q\in(0,1]$, and
$Q\in\mathcal{S}^n_{o}$ be origin-symmetric.
A non-trivial, even, finite, Borel
measure $\mu$  on $S^{n-1}$ is the weighted dual curvature measure $\widetilde{C}_{q}(K,Q;\cdot)$
of an origin-symmetric convex body
$K$ in $\R^n$ if and only if $\mu$ is not concentrated on any great subsphere.
\end{coro}

Our second result is if $1<q<n$ and the parameter body $Q\in\mathcal{S}^n_{o}$ is convex, then
the subspace concentration property \eqref{subspace-concentration} characterizes $G$-invariant
weighted dual curvature measures as well.

\begin{theo}
\label{q0n-G-weighted}
Let $n\geq 2$, $q\in(1,n)$, $G\subset O(n)$ a closed subgroup without a non-zero fixed point,
and $Q\in\mathcal{K}^n_{o}$ be $G$-invariant.
Given a non-trivial, $G$-invariant, finite, Borel
measure $\mu$ on $S^{n-1}$, the weighted dual curvature measure equation,
\[
\widetilde{C}_{q}(K,Q;\cdot) = \mu,
\]
has a $G$-invariant convex body solution
$K$ if and only if $\mu$  is not concentrated on any great subsphere, and
\begin{equation}
\label{subspace-concentration-G}
\frac{\mu(L\cap S^{n-1})}{\mu(S^{n-1})} < \frac{{\rm dim}\,L}q
\end{equation}
for any proper $G$-invariant linear subspace $L\subset \R^n$.
\end{theo}

\begin{coro}
\label{q0n-even-weighted}
Let $n\geq 2$, $q\in(0,n)$,
$Q\in\mathcal{K}^n_{o}$ be origin-symmetric.
Given a non-trivial, even, finite, Borel measure $\mu$ on $S^{n-1}$,
the weighted dual curvature measure equation,
\[
\widetilde{C}_{q}(K,Q;\cdot)=\mu,
\]
has an origin-symmetric convex body solution $K$ if and only if $\mu$ is not concentrated
on any great subsphere, and
\begin{equation}
\label{subspace-concentration-even}
\frac{\mu(L\cap S^{n-1})}{\mu(S^{n-1})} < \frac{{\rm dim}\,L}q
\end{equation}
for any proper linear subspace $L\subset \R^n$.
\end{coro}

For a general star body $Q\in\mathcal{S}^n_{o}$, we have the following sufficient condition for any $q>1$:

\begin{theo}
\label{q0-G-weighted-sufficient}
Let $n\geq 2$, $q>1$, $G\subset O(n)$ a closed subgroup without a non-zero fixed point, and
$Q\in\mathcal{S}^n_{o}$ be $G$-invariant.
Given a non-trivial, $G$-invariant, finite, Borel
measure $\mu$  on $S^{n-1}$,  the weighted dual curvature measure equation,
\[
\widetilde{C}_{q}(K,Q;\cdot) = \mu,
\]
has a $G$-invariant convex body solution
$K$ if $\mu$ is not concentrated on any great subsphere and  satisfies
\begin{equation}
\label{subspace-concentration-G-sufficient}
\frac{\mu(L\cap S^{n-1})}{\mu(S^{n-1})} < \frac{{\rm dim}\,L}q
\end{equation}
for any proper $G$-invariant linear subspace $L\subset \R^n$.
\end{theo}

\begin{coro}
\label{q0-even-weighted-sufficient}
Let $n\geq 2$, $q>1$, and
$Q\in\mathcal{S}^n_{o}$ be origin-symmetric.
Given a non-trivial, even, finite, Borel
measure $\mu$  on $S^{n-1}$, the weighted dual curvature measure equation,
\[
\widetilde{C}_{q}(K,Q;\cdot) = \mu,
\]
has an origin-invariant convex body solution $K$ if $\mu$
is not concentrated on any great subsphere and satisfies
\begin{equation}
\label{subspace-concentration-even-sufficient}
\frac{\mu(L\cap S^{n-1})}{\mu(S^{n-1})} < \frac{{\rm dim}\,L}q
\end{equation}
for any proper linear subspace $L\subset \R^n$.
\end{coro}

On the other hand, extending the estimate due to Henk-Pollehn \cite{HeP18}, we will show a concentration property
for the dual curvature measure $\widetilde C_q(K,Q; \cdot)$ when $Q$ is a polar $L_{q-n}$ zonoid, $q\ge n+1$.
We recall that for  $p\ge 1$ and
a finite even Borel measure $\mu$ on $S^{n-1}$ not concentrated on any great subsphere,
the $L_p$ zonoid $Z_p(\mu)$ associated with $\mu$ (see e.g. Ball \cite{Bal91a,Bal91b}, Barthe \cite{Bar98,Bar04},
 Campi-Gronchi \cite{CaG06}, Dai-Wang \cite{DaT26+}, Lutwak-Yang-Zhang \cite{LYZ00}, Milman-Shabelman-Yehudayoff \cite{MSY25},
 Schneider-Weil \cite{ScW83}) is defined by
$$
h_{Z_p(\mu)}(v)^p=\int_{S^{n-1}}|\langle u,v\rangle|^p\,d\mu(u).
$$
We note that any such even measure $\mu$ is the $L_p$ surface area measure of some origin-symmetric convex body $K\subset \R^n$ (cf. Lutwak \cite{Lut93a}), and then $Z_p(\mu)$ is the $L_p$ projection body $\Pi_pK$ of $K$.
In addition,
$Z_p^*(\mu)$ is the polar $L_p$ zonoid defined by the formula
\begin{equation}
\label{polarLp-zonoid}
\|x\|_{Z_p^*(\mu)}^p=\int_{S^{n-1}}|\langle x,u\rangle|^p\,d\mu(u),\qquad x\in\R^n,
\end{equation}
which is just the polar of $Z_p(\mu)$ in the classical sense.
It follows from Lewis \cite{Lew78} (see also
 Lutwak-Yang-Zhang \cite{LYZ04,LYZ05}) that any $n$-dimensional
subspace of $L_p$ is isometric to
$\|\cdot\|_{Z_p^*(\mu)}$ for some isotropic even measure $\mu$ on $S^{n-1}$.

\begin{theo}
\label{qn+1-Lpzonoid-weighted}
Let $n\geq 2$ and $q\geq n+1$. If $K,Q\in\mathcal{K}^n_{o}$ where $Q$ is a polar $L_{q-n}$ zonoid, then
\begin{equation}
\label{subspace-concentration-G}
\frac{\widetilde{C}_{q}(K,Q;L\cap S^{n-1})}{\widetilde{C}_{q}(K,Q;S^{n-1})} < \frac{q-n+{\rm dim}\,L}q
\end{equation}
for any proper linear subspace $L\subset \R^n$.
\end{theo}

\section{Preliminaries}
For more detailed information on the theory of convex bodies, we refer to  Gruber \cite{Gru07} and Schneider \cite{Sch14}.

Our setting is the Euclidean $n$-space $\R^n$ for $n\geq 2$. We write $o$ to denote the origin in $\R^n$, $\langle\cdot,\cdot\rangle$ for the standard inner product, and $\|\cdot\|$ for its induced norm. We denote the unit ball  by $B^n=\{x\in\R^n:\|x\|\leq 1\}$ and the unit sphere by $S^{n-1}=\partial B^n$. $\HH^k(\cdot)$ stands for the $k$-dimensional Hausdorff measure
 normalized in a way that it coincides with the Lebesgue measure on $\R^k$, and we use the notation $V(\cdot)$ for the $n$-dimensional volume (Lebesgue measure). In particular, the volume of the unit ball is denoted as $\kappa_n=V(B^n)=\frac{\pi^{\frac n2}}{\Gamma(\frac n2 +1)}$ and its surface area is  $\HH^{n-1}(S^{n-1})=n\kappa_n$, where $\Gamma$ is Euler's gamma function. We call a compact convex set $K\subset\R^n$ with non-empty interior a convex body.  We use the symbol  $\mathcal{K}^n_{o}$ to denote the family of all convex bodies $K$ which contain $o$ in their interior, that is, $o\in{\rm int}\,K$.

For a convex body $K\subset\R^n$, the support function $h_K(u):S^{n-1}\to \R$ is defined as $h_K(u)=\max \{\langle x,u\rangle : x\in K\}$. One observes that
\begin{equation}
\label{SupportFunctionTranslate}
h_{K-z}(u)=h_K(u)-\langle u,z\rangle
\end{equation}
for any compact convex set $K\subset\R^n$ and $z,u\in\R^n$. If $u\in S^{n-1}$, then the face of $K$ with exterior unit normal $u$ is given by
$F(K,u)=\{x\in K: \langle x,u\rangle=h_K(u) \}$.
For $x\in{\partial} K$, let the spherical image of $x$ be defined as
${\pmb\nu}_K(\{x\})=\{u\in S^{n-1}: h_K(u)=\langle x,u\rangle\}$. For a Borel set $\eta\subset S^{n-1}$, the reverse spherical image is defined as
$$
{\pmb\nu}_K^{-1}(\eta)=\{x\in{\partial} K:\, {\pmb\nu}_K(x)\cap \eta\neq \emptyset\}=\cup_{u\in\eta}F(K,u).
$$
If $K$ has a unique supporting hyperplane at $x$, then we say that $K$ is smooth at $x$, and in this case ${\pmb\nu}_K(\{x\})$ contains exactly one element that we denote by $\nu_K(x)$ and call it the exterior unit normal of $K$ at $x$.

\section{The dual intrinsic volume and dual curvature measure}

Following Huang-Lutwak-Yang-Zhang \cite{HLYZ16} and Lutwak-Yang-Zhang \cite{LYZ18}, if $K\in\mathcal{K}^n_{o}$ and $\eta\subset S^{n-1}$ is a Borel set,
then the reverse radial Gauss image of $\eta$ is given by
\begin{align*}
{\pmb\alpha}^*_K(\eta)&=\{u\in S^{n-1}: \varrho_K(u)u\in F(K,v)\text{ for some }v\in\eta\}\\
&=
\{u\in S^{n-1}: \varrho_K(u)u\in{\pmb\nu}_K^{-1}(\eta)\}.
\end{align*}
It is Lebesgue measurable according to \cite[Lemma~2.2.4]{Sch14}.
 For a convex body $K\in\mathcal{K}_{o}^n$ and $q\neq 0$, the weighted $q$th dual curvature measure $\widetilde{C}_q(K,Q;\cdot)$ with respect to $Q\in\mathcal{S}_o^n$ is a Borel measure on $S^{n-1}$ and is defined in the paper by Lutwak-Yang-Zhang  \cite{LYZ18} (cf. \cite{HLYZ16} in the case $Q=B^n$) as
\begin{equation}\label{dualcurvmeasure}
\widetilde{C}_q(K,Q;\eta)=\frac1n\int_{{\pmb\alpha}^*_K(\eta)}\varrho_K^{q}(u)\varrho^{n-q}_Q(u)\HH^{n-1}(du),
\end{equation}
and hence
$$
\widetilde{V}_q(K,Q)=\widetilde{C}_q(K,Q;S^{n-1}).
$$
In addition, for a Borel set $\eta\subset S^{n-1}$, using the Borel set
$$
\Xi(\eta)=\bigcup\{{\rm conv}\{o,x\}:\,x\in {\pmb\nu}_K^{-1}(\eta)\}
$$
and polar coordinates yields that
\begin{equation}
\label{dual-curvature-integral}
\widetilde{C}_q(K,Q;\eta)=\frac{q}n\int_{\Xi(\eta)}\|x\|_Q^{q-n}\,dx.
\end{equation}

Based on work by Huang-Lutwak-Yang-Zhang \cite{HLYZ16}, Lutwak-Yang-Zhang \cite{LYZ18}
proved the variational formula for the weighted dual intrinsic volume. Namely,
for $q\neq 0$, $K\in\mathcal{K}_{o}^n$, $Q\in\mathcal{S}_o^n$, and
continuous $\varphi:S^{n-1}\to\R$, if $K_t$ is the corresponding Wulff-shape
$$
K_t=\{x\in\R^n:\,\langle x,u\rangle\leq h_K(u)+t\varphi(u)\,\;\forall u\in S^{n-1}\}
$$
when $|t|$ is small, then
\begin{equation}
\label{dualAlexandrov}
\lim_{t\to 0}\frac{\widetilde{V}_q(K_t,Q)-\widetilde{V}_q(K,Q)}{t}=q\int_{S^{n-1}}\frac{\varphi(u)}{h_K(u)}\,d \widetilde{C}_q(K,Q;u).
\end{equation}
According to  Lemma~5.1 in \cite{LYZ18},
 if $q\neq 0$ and the Borel function $g:\,S^{n-1}\to \R$ is bounded, then
\begin{equation}
\label{intgCqoin}
\int_{S^{n-1}}g(u)\,d\widetilde{C}_{q}(K,Q;u)
=\frac1n\int_{\partial' K} g(\nu_K(x))\langle \nu_K(x),x\rangle\|x\|_Q^{q-n}\,d\HH^{n-1}(x)
\end{equation}
where $\|x\|_Q=\min\{\lambda\geq 0:\, x\in \lambda Q\}$ is a measurable and $1$-homogeneous function
satisfying $\|x\|_Q>0$ for $x\neq o$.
The advantage of introducing the star body $Q$ is also apparent in the  equiaffine invariant formula
(see Theorem~6.8 in \cite{LYZ18}), stating that
if $\Phi\in{\rm SL}(n,\R)$, then
\begin{equation}
\label{Cqaffineinv0}
\int_{S^{n-1}}g(u)\,d\widetilde{C}_{q}(\Phi K,\Phi Q;u)=
\int_{S^{n-1}} g\left(\frac{\Phi^{-t} u}{\|\Phi^{-t} u\|}\right)d\widetilde{C}_{q}( K,Q;u),
\end{equation}
{where $\Phi^{-t}$ denotes the transpose of the inverse of $\Phi$.}

Lutwak-Yang-Zhang  \cite{LYZ18} proved
that the $L_p$ $q$th weighted curvature measure is also weakly continuous.

\begin{lemma}[Lutwak, Yang, Zhang  \cite{LYZ18}]
\label{CpqcontQ}
For $q\neq 0$, $p\in\R$ and $Q\in\mathcal{S}_o^n$, if $\{K_m\}_{m\in\N}$ tends to $K$ for
compact convex sets $K_m,K\subset\R^n$ containing $o$, then
$\widetilde{V}_{q}(K_m,Q)$ tends to $\widetilde{V}_{q}(K,Q)$, and
$\widetilde{C}_{p,q}(K_m,Q;\cdot)$ tends weakly to $\widetilde{C}_{p,q}(K,Q;\cdot)$.
\end{lemma}

\section{An upper bound for the dual intrinsic volume}
\label{secDualIntVolumeBS}

In this section, we survey some results and ideas in
Haodi Chen's paper \cite{HaodiChen}, providing the upper bound (see Proposition~\ref{qthIntrinsicBox})
on  the $q$th dual intrinsic volume for $q>0$.
The following essentially optimal result improves earlier estimates by  Yiming Zhao \cite{Zha18} and
 B\"or\"oczky-Lutwak-Yang-Zhang-Zhao \cite{BLYZ19}.

\begin{prop}[Haodi Chen]
\label{qthIntrinsicBox}
For $n\geq 2$ and  $q>0$, there exists $\widetilde{C}(n,q)>1$ depending only on $n,q$ with the following properties:
If $0<a_1\leq\ldots\leq a_n$ and
$e_1,\ldots,e_n$ form an orthonormal basis of $\R^n$,
 then the rectangular box $R=\sum_{i=1}^n[-a_1,a_1]e_i$
satisfies that
$$
\widetilde{V}_q(R)\leq \widetilde{C}(n,q)\cdot
\left\{
\begin{array}{rl}
a_1^q&\mbox{if }0<q<1 \\
a_1\ldots a_i a_{i+1}^{q-i}&\mbox{if } i<q<i+1,\;i=1,\ldots,n-2\\
a_1\ldots a_{n-1} a_{n}^{q-n+1}&\mbox{if } q>n-1\\
a_1\ldots a_q\left(1+\log \frac{a_{q+1}}{a_q}\right)&\mbox{if }q=1,\ldots,n-1
\end{array}.
\right.
$$
In addition, replacing $\widetilde{C}(n,q)$ by $\widetilde{C}(n,q)^{-1}$, we obtain a lower bound
for $\widetilde{V}_q(R)$.
\end{prop}

In the rest of the section, we discuss some ways how to use Proposition~\ref{qthIntrinsicBox}.
Let  $G\subset O(n)$ be a closed subgroup without a non-zero fixed point, and let $K\in\mathcal{K}^n_{o}$ be invariant under $G$.
We write $E\subset K$ to denote the unique maximal volume John ellipsoid contained in $K$ (cf. Schneider \cite{Sch14}). Being unique, $E$ is also invariant under $G$, acting without a non-zero fixed point, and hence the origin $o$ is the center of $E$. It follows that (cf. Schneider \cite{Sch14})
\begin{equation}
\label{EellipsoidnE}
E\subset K\subset nE.
\end{equation}
In particular, if $e_1,\ldots,e_n$ form the
orthonormal basis  of $\R^n$ corresponding to the principal directions of $E$,  and
$a_1,\ldots,a_n>0$ are the corresponding half axes of $E$, then
the rectangular box
$\Gamma=\sum_{i=1}^n[-a_i,a_i]e_i$ satisfies that
\begin{equation}
\label{RboxBigR}
 n^{-1}\Gamma\subset K\subset n\Gamma.
\end{equation}
A consequence of Proposition~\ref{qthIntrinsicBox} is a simple estimate about the inradius, which follows from \eqref{RboxBigR}, Proposition~\ref{qthIntrinsicBox}, and the $q$-homogeneity of the $q$th intrinsic volume.

\begin{lemma}
\label{inradiusVq}
For $q>0$, $R,c>1$, and $n\geq 2$, there exists $\xi=\xi(n,q,R,c)>0$ such that if $K\subset RB^n$ is a centered convex body,
$Q\in\mathcal{S}_o^n$ with $c^{-1}B^n\subset Q\subset cB^n$,   and $\widetilde{V}_q(K,Q)\geq t$ for $t>0$, then  $\xi\cdot t^{\frac1q}B^n\subset K$.
\end{lemma}

\section{Proofs of Theorem~\ref{q01-G-weighted} and Theorem~\ref{q0-G-weighted-sufficient}}

As in Theorem~\ref{q01-G-weighted} and Theorem~\ref{q0-G-weighted-sufficient}, we may assume that $\mu$ is a probability measure by rescaling. It then suffices to prove following theorem.

\begin{theo}
\label{q0-G-weighted-sufficient0}
Let $n\geq 2$, $q>0$, and let $G\subset O(n)$ be a closed subgroup without a non-zero fixed point. Let also
$Q\in\mathcal{S}^n_{o}$ be invariant under $G$.
If a $G$-invariant Borel probability
measure $\mu$  on $S^{n-1}$
 is not concentrated on any great subsphere and  satisfies
\begin{equation}
\label{subspace-concentration-G-sufficient0}
\mu(L\cap S^{n-1})<\mbox{$\frac{{\rm dim}\,L}q$}
\end{equation}
for any proper $G$-invariant linear subspace $L\subset \R^n$, then $\mu=\widetilde{C}_{q}(K,Q;\cdot)$ for a $G$-invariant convex body $K\subset\R^n$.
\end{theo}

Let us introduce the notation that will be used throughout the whole section. Let $\mu$ be a $G$-invariant Borel probability
measure   on $S^{n-1}$ that
 is not concentrated onto any great subsphere and  satisfies \eqref{subspace-concentration-G-sufficient0}.
For $\delta\in(0,1)$ and a linear $j$-dimensional subspace $L\subset\R^n$, $j=1,\ldots,n-1$, let
$$
\Psi(L\cap S^{n-1},\delta)=\left\{u\in S^{n-1}:\,|\langle u,v\rangle|\leq\delta,\;\forall v\in L^\bot\cap S^{n-1}\right\}.
$$
For any $j=1,\ldots,n-1$, the space of $j$-dimensional linear subspaces is compact, and $G\subset O(n)$ is a closed subgroup. We deduce the existence of  $\tau,\delta\in(0,1)$ such that  for any $G$-invariant linear $j$-dimensional subspace $L\subset\R^n$, we have
\begin{equation}
\label{subspace-concentration-tau}
\mu\left( \Psi(L\cap S^{n-1},\delta)\right)<\frac{j(1-\tau)}{q}.
\end{equation}
In addition, since $\mu$ is not concentrated on any great subsphere and is invariant under $G$ with no non-zero fixed point, we deduce that the $\mu$-measure of any open hemisphere is positive.
Therefore, we may also assume about $\delta$ and $\tau$ that for any $v\in S^{n-1}$, we have
\begin{equation}
\label{subspace-concentration-tau-n-1}
\mu\left( \left\{u\in S^{n-1}:\,\langle u,v\rangle\geq \delta\right\}\right)>\tau.
\end{equation}

 If $C\in\mathcal{K}_o^n$, then we consider the entropy type function
$$
\mathcal{E}(C)=
\int_{S^{n-1}}\log h_{C}\,d\mu-\frac1q\log\widetilde{V}_q(C,Q),
$$
which is invariant under rescaling. Namely, it satisfies
\begin{equation}
\label{entropy-rescale-invariant}
\mathcal{E}(\lambda\,C)=\mathcal{E}(C) \mbox{ \ for $\lambda>0$}.
\end{equation}
In addition, the entropy is continuous. In particular, Lemma~\ref{CpqcontQ} yields that if $C_m\in\mathcal{K}_o^n$ tends to $C\in\mathcal{K}_o^n$, then
\begin{equation}
\label{entropy-continuous}
\lim_{m\to\infty}\mathcal{E}(C_m)=\mathcal{E}(C).
\end{equation}

Let $\mathcal{C}$ be the set of $G$-invariant convex bodies $C\subset\R^n$ with $\widetilde{V}_q(C,Q)=1$.
 Any $C\in\mathcal{C}$ is centered, since $C$ is $G$-invariant, $G$ has no non-zero fixed point, and the centroid is equivariant with respect to linear transformations. We observe that if $C\in\mathcal{C}$, then
$$
\mathcal{E}(C)=
 \int_{S^{n-1}}\log h_{C}\,d\mu.
$$
In order to prove Theorem~\ref{q0-G-weighted-sufficient0}, the first step is to show the following.

\begin{prop}
\label{q0-G-weighted-sufficient-limit}
For $n\geq 2$, $q>0$, closed subgroup $G\subset O(n)$  without a non-zero fixed point, let
$Q\in\mathcal{S}^n_{o}$ be invariant under $G$, and and measure $\mu$ satisfy the same conditions as in Theorem~\ref{q0-G-weighted-sufficient0}.
 Using the notation as above, if $C_m\in\mathcal{C}$ satisfies that $\lim_{m\to\infty}{\rm diam}C_m=\infty$, then
\begin{equation}
\label{Cmdiaminfinity}
\lim_{m\to\infty}\mathcal{E}(C_m)=\infty.
\end{equation}
\end{prop}
\proof Let $e^{(m)}_1,\ldots,e^{(m)}_n\in S^{n-1}$ be an orthonormal basis of $\R^n$ forming the principal directions associated to the $G$-invariant ellipsoid $E$ in
\eqref{EellipsoidnE} with half axes  $a^{(m)}_1,\ldots,a^{(m)}_n>0$,
 where we may assume that $a^{(m)}_1\leq \ldots\leq a^{(m)}_n$. In particular,
\begin{align}
\label{aieninE}
\pm a^{(m)}_ie^{(m)}_i&\in C_m\mbox{ \ for }i=1,\ldots,n\\
\label{a1BninE}
a^{(m)}_1 B^n&\subset C_m.
\end{align}
It also follows that $C_m$ and the rectangular box
$\Gamma_m=\sum_{i=1}^n[-a^{(m)}_i,a^{(m)}_i]e^{(m)}_i$ satisfy
\begin{equation}
\label{RboxBigR}
 n^{-1}\Gamma_m\subset C_m\subset n\Gamma_m.
\end{equation}
Since $\lim_{m\to\infty}{\rm diam}\,C_m=\infty$, we have $\lim_{m\to\infty}a^{(m)}_n=\infty$ by \eqref{RboxBigR}. And since $1=\widetilde{V}_q(K_m,Q)\geq \widetilde{V}_q(a^{(m)}_1B^n,Q)$ (cf. \eqref{a1BninE}), we deduce that the sequence $\{a^{(m)}_1\}$ is bounded.
Therefore, after possibly taking a subsequence, we may assume  that for some $\tilde{k}\in\{1,\ldots,n-1\}$ and $A>0$,
\begin{equation}
\label{a1boundedaninfty}
\begin{array}{rcll}
\lim_{m\to\infty}a^{(m)}_i&=&\infty&\mbox{ if }i>\tilde{k}\\
a^{(m)}_i&\leq &A&\mbox{ if }i\leq \tilde{k}.
\end{array}
\end{equation}
We may also assume that $e_1,\ldots,e_n$ are indexed respecting the $G$-invariance in the following way.
For some $\ell\in\{2,\ldots,n\}$ and $1=k_1<\ldots<k_{\ell}\leq n$, $k_{\ell+1}=n+1$, if
$\alpha=1,\ldots,\ell$, then let
\begin{itemize}
\item   $\widetilde{L}^{(m)}_\alpha={\rm lin}\{e^{(m)}_{k_\alpha},\ldots,e^{(m)}_{k_{\alpha+1}-1}\}$ be an irreducible $G$-invariant linear subspace;
\item $d_\alpha=k_{\alpha+1}-k_\alpha$ .
\end{itemize}
It follows by the $G$-invariance of $E$ that $a^{(m)}_{i}=a^{(m)}_{k_\alpha}$ if $k_\alpha\leq i<k_{\alpha+1}$, $\alpha=1,\ldots,\ell$.

We choose $\theta>1$ such that $\varrho_Q(u)^{n-q}\leq \theta$ for $u\in S^{n-1}$. Write $\lfloor q\rfloor$ ($\lceil q\rceil$) to denote the largest integer not larger than $q$ (the smallest integer not smaller than $q$), and set $p=\min\{\lfloor q\rfloor,n-1\}$.
We deduce from $\widetilde{V}_q(C_m)=1$,  Proposition~\ref{qthIntrinsicBox}, and \eqref{RboxBigR} that
\begin{equation}
\label{tildeVqCmpup}
\aleph\leq
\left\{
\begin{array}{rll}
(a^{(m)}_1)^q&\mbox{if }&0<q<1 \\
a^{(m)}_1\ldots a^{(m)}_p( a^{(m)}_{p+1})^{q-p}&\mbox{if }& q>1,\\
&&q\not\in\{1,\ldots,n-1\}\\
a^{(m)}_1\ldots a^{(m)}_p\left(1+\log \frac{a^{(m)}_{p+1}}{a^{(m)}_p}\right)&\mbox{if }&q=1,\ldots,n-1
\end{array},
\right.
\end{equation}
where $\aleph=\widetilde{C}(n,q)^{-1} \theta^{-1} n^{-q}>0$ is independent of $m$. We choose $\gamma\in\{1,\ldots,\ell\}$ such that
$$
k_\gamma\leq \min\{\lceil q\rceil,n\}<k_{\gamma+1}.
$$
It follows that if $q\geq 1$, then
\begin{equation}
\label{tildeVqCmpupG}
\aleph\leq \left\{
\begin{array}{ll}
(a^{(m)}_{k_\gamma})^{q-k_\gamma+1}\prod_{1\leq \alpha<\gamma}(a^{(m)}_{k_\alpha})^{d_\alpha}&\mbox{if } q\not\in\{1,\ldots,n-1\}\\[1ex]
(a^{(m)}_{k_\gamma})^{q-k_\gamma+1}
\left(1+\log \frac{a^{(m)}_{q+1}}{a^{(m)}_q}\right)
\prod_{1\leq \alpha<\gamma}(a^{(m)}_{k_\alpha})^{d_\alpha}&\mbox{if }q\in\{1,\ldots,n-1\}
\end{array}, \right.
\end{equation}
where the product over $1\leq \alpha<\gamma$ is void if $\gamma=1$.

For $\Theta_m=\left\{u\in S^{n-1}:\,\langle u,e^{(m)}_n\rangle\geq \delta\right\}$,  it follows from
\eqref{subspace-concentration-tau-n-1}, \eqref{aieninE}, and \eqref{a1boundedaninfty} that
\begin{equation}
\label{tildeECmThetainfty}
\lim_{m\to\infty}\int_{\Theta_m}\log h_{C_m}\,d\mu\geq \lim_{m\to\infty} (\log a^{(m)}_n)\cdot\delta\cdot \tau
=\infty.
\end{equation}

\noindent{\bf Case 1:} $q\in(0,1)$.

It follows from  \eqref{tildeVqCmpup} that $a^{(m)}_1>\aleph^{\frac1q}$, and hence \eqref{tildeECmThetainfty} yields that
$$
\lim_{m\to\infty}\mathcal{E}(C_m)\geq \lim_{m\to\infty}\int_{\Theta_m}\log h_{C_m}\,d\mu+
 \int_{S^{n-1}\backslash \Theta_m}\log \aleph^{\frac1q}\,d\mu=\infty.
$$

\noindent{\bf Case 2:} $1\leq q<n-1$ and there exists $A_0>1$ such that $a^{(m)}_1,\ldots,a^{(m)}_p<A_0$,
$a^{(m)}_{p+1}<A_0$ if $q\not\in\N$, and $\frac{a^{(m)}_{p+1}}{a^{(m)}_p}<A_0$ if $q\in\N$.

In this case, there exists $a>0$ independent of $m$ such that $a^{(m)}_1>a$    by \eqref{tildeVqCmpup}. Therefore,
\eqref{tildeECmThetainfty} yields that
$$
\lim_{m\to\infty}\mathcal{E}(C_m)\geq \lim_{m\to\infty}\int_{\Theta_m}\log h_{C_m}\,d\mu+
 \int_{S^{n-1}\backslash \Theta_m}\log a\,d\mu=\infty.
$$

\noindent{\bf Case 3:} $ q>1$, $q\not\in\{1,\ldots,n-1\}$
and  $\lim_{m\to\infty}a^{(m)}_{p+1}=\infty$, or equivalently, $\lim_{m\to\infty}a^{(m)}_{k_\gamma}=\infty$.

In particular, \eqref{a1boundedaninfty} yields that this is automatically the case if $q>n-1$.

In Case 3, \eqref{a1boundedaninfty} yields that $\gamma\geq 2$.
We observe that for any $v\in S^{n-1}$ and $C_m$, there exists an $e^{(m)}_i$ such that
$|\langle v,e^{(m)}_i\rangle|\geq \frac{1}{\sqrt{n}}> \frac{\delta}{n}$.
For each $C_m$
 and $\alpha=1,\ldots,\gamma-1$, we define
\begin{align*}
B^{(m)}_\alpha=&\left\{v\in S^{n-1}:\,\exists i\in\{k_\alpha,\ldots,k_{\alpha+1}-1\},\;
|\langle v,e^{(m)}_i\rangle|\geq \frac{\delta}{n}\right.\\
&\left.\mbox{ \ \ \ \ \ \ \ \ \ \ and }
|\langle v,e^{(m)}_j\rangle|<\frac{\delta}{n}\mbox{ for }j\geq k_{\alpha+1}\right\},
\end{align*}
and let
$$
B^{(m)}_{\gamma}=\left\{v\in S^{n-1}:\,\exists i\geq k_\gamma,\;
|\langle v,e^{(m)}_i\rangle|\geq \frac{\delta}{n}\right\}.
$$
By
\eqref{aieninE},
 if $\alpha=1,\ldots,\gamma$ and $v\in B^{(m)}_{\alpha}$, then
\begin{equation}
\label{hCmBalpha}
h_{C_m}(v)\geq \frac{\delta}{n}\cdot a^{(m)}_{k_\alpha}.
\end{equation}
For $\alpha=1,\ldots,\gamma-1$, we consider the $G$-invariant subspace
$$
L^{(m)}_\alpha=\sum_{\eta=1}^\alpha\widetilde{L}^{(m)}_\eta
\mbox{ \ with }{\rm dim}\,L^{(m)}_\alpha=\sum_{\eta=1}^\alpha d_\eta,
$$
and hence if $\eta\leq \alpha$, then
\begin{equation}
\label{BjPsiLi}
B^{(m)}_\eta\subset \Psi(L^{(m)}_\alpha\cap S^{n-1},\delta).
\end{equation}
We note that \eqref{tildeVqCmpupG} reads as
\begin{equation}
\label{tildeVqCmpupG0}
\aleph\leq
(a^{(m)}_{k_\gamma})^{q-\sum_{\alpha=1}^{\gamma-1}d_\alpha}\prod_{\alpha=1}^{\gamma-1}(a^{(m)}_{k_\alpha})^{d_\alpha}.
\end{equation}

It follows that $S^{n-1}$ is partitioned into the Borel sets $B^{(m)}_1,\ldots,B^{(m)}_{\gamma}$, and
\eqref{BjPsiLi} yields that if $\alpha=1,\ldots,\gamma-1$, then
\begin{eqnarray}
\label{muAi}
\mu(B^{(m)}_1)+\ldots+\mu(B^{(m)}_\alpha)&\leq &\frac{(1-\tau)\sum_{\eta=1}^\alpha d_\eta}{q},\\
\label{muAn}
\mu(B^{(m)}_1)+\ldots+\mu(B^{(m)}_{\gamma})&=&1.
\end{eqnarray}
For $\zeta=\frac{1-\tau}{q}$, we have $0< \zeta<\frac1{q}$, and we define
\begin{eqnarray}
\label{betai}
\beta_\alpha&= &\mu(B_\alpha)-d_\alpha\zeta\mbox{ \ for $\alpha=1,\ldots,\gamma-1$}\\
\label{betan}
\beta_{\gamma}&=&\mu(B_{\gamma})-\left(q-\sum_{\alpha=1}^{\gamma-1}d_\alpha\right)\zeta-\tau.
\end{eqnarray}
Notice that (\ref{muAi}) and (\ref{muAn}) yield
\begin{eqnarray}
\label{sumbetai}
\beta_1+\ldots+\beta_\alpha&\leq &0\mbox{ \ for $\alpha=1,\ldots,\gamma-1$}\\
\label{sumbetan}
\beta_1+\ldots+\beta_{\gamma}&=&0.
\end{eqnarray}
We deduce from applying \eqref{hCmBalpha}, \eqref{muAn},  \eqref{betai}, and \eqref{betan}
that
\begin{align*}
\int_{S^{n-1}}\log h_{C_m}\,d\mu= &
\sum_{\alpha=1}^{\gamma}\int_{B^{(m)}_\alpha}\log h_{C_m}\,d\mu\\
\geq &\sum_{\alpha=1}^{\gamma}\mu(B^{(m)}_\alpha)\log a^{(m)}_{k_\alpha}+\sum_{\alpha=1}^{\gamma}\mu(B^{(m)}_\alpha)\log \frac{\delta}{n}\\
= & \sum_{\alpha=1}^{\gamma}\mu(B^{(m)}_\alpha)\log a^{(m)}_{k_\alpha}+\log \frac{\delta}{n}\\
=&
\sum_{\alpha=1}^{\gamma}\beta_\alpha\log a^{(m)}_{k_\alpha}+\sum_{\alpha=1}^{\gamma-1}d_\alpha\zeta \log a^{(m)}_{k_\alpha}+\left(q-\sum_{\alpha=1}^{\gamma-1}d_\alpha\right)\zeta\log a^{(m)}_{k_\gamma}\\
&+\tau\log a^{(m)}_{k_\gamma}+\log \frac{\delta}{n}.
\end{align*}
In turn, it follows from
\eqref{tildeVqCmpupG0}, \eqref{sumbetai}, \eqref{sumbetan},
and the fact that $a^{(m)}_{k_\alpha}\leq a^{(m)}_{k_{\alpha+1}}$ that
\begin{align*}
\int_{S^{n-1}}\log h_{C_m}\,d\mu\geq &
\sum_{\alpha=1}^{\gamma}\beta_\alpha\log a^{(m)}_{k_\alpha}+\zeta\log \aleph
+\tau\log a^{(m)}_{k_\gamma}+\log \frac{\delta}{n}\\
= &
(\beta_1+\ldots+\beta_\gamma)\log a^{(m)}_{k_\gamma}+\\
&+\sum_{\alpha=1}^{\gamma-1}(\beta_1+\ldots+\beta_\alpha)\left(\log a^{(m)}_{k_\alpha}-\log a^{(m)}_{k_{\alpha+1}}\right)+\\
& +\zeta\log \aleph
+\tau\log a^{(m)}_{k_\gamma}+\log \frac{\delta}{n}\\
\geq &\zeta\log \aleph
+\tau\log a^{(m)}_{k_\gamma}+\log \frac{\delta}{n}.
\end{align*}
As $\lim_{m\to\infty}a^{(m)}_{k_\gamma}=\infty$, we conclude that
$\lim_{m\to\infty}\int_{S^{n-1}}\log h_{C_m}\,d\mu=\infty$.\\

\noindent{\bf Case 4:} $q\in\{1,\ldots,n-1\}$, and either $\lim_{m\to\infty}a^{(m)}_{q}=\infty$ or
 $\lim_{m\to\infty}\frac{a^{(m)}_{q+1}}{a^{(m)}_q}=\infty$.

After possibly taking a subsequence, we may assume that either
\begin{description}
\item{(i)}  there exists $A_1>1$ such that $\frac{a^{(m)}_{q+1}}{a^{(m)}_q}\leq A_1$ for each $C_m$,
\item{(ii)} or $\lim_{m\to\infty}\frac{a^{(m)}_{q+1}}{a^{(m)}_q}=\infty$.
\end{description}
If $\frac{a^{(m)}_{q+1}}{a^{(m)}_q}\leq A_1$ for each $C_m$, then $\lim_{m\to\infty}a^{(m)}_{q}=\infty$, and hence
$\gamma\geq 2$ and $\lim_{m\to\infty}a^{(m)}_{k_\gamma}=\infty$. In addition, we deduce from that \eqref{tildeVqCmpupG} that
\begin{equation}
\label{tildeVqCmpupG1}
(1+\log A_1)^{-1}\aleph\leq
(a^{(m)}_{k_\gamma})^{q-\sum_{\alpha=1}^{\gamma-1}d_\alpha}\prod_{\alpha=1}^{\gamma-1}(a^{(m)}_{k_\alpha})^{d_\alpha}.
\end{equation}
Thus, we can apply the same argument in Case~3, only replacing the $\aleph$ in \eqref{tildeVqCmpupG0} by the $(1+\log A_1)^{-1}\aleph$ in \eqref{tildeVqCmpupG1} to conclude that $\lim_{m\to\infty}\int_{S^{n-1}}\log h_{C_m}\,d\mu=\infty$.

Therefore, we assume that $\lim_{m\to\infty}\frac{a^{(m)}_{q+1}}{a^{(m)}_q}=\infty$, as in (ii). In this case, we have
$$
k_{\gamma+1}=q+1 \mbox{ and }q=\sum_{\alpha=1}^{\gamma}d_\alpha,
$$
and we apply a variant of the argument in Case~3.

Again, for any $v\in S^{n-1}$ and $C_m$, there exists $e^{(m)}_i$ such that
$|\langle v,e^{(m)}_i\rangle|\geq \frac{1}{\sqrt{n}}> \frac{\delta}{n}$.
For each $C_m$
 and $\alpha=1,\ldots,\gamma$, we define
\begin{align*}
B^{(m)}_\alpha=&\left\{v\in S^{n-1}:\,\exists i\in\{k_\alpha,\ldots,k_{\alpha+1}-1\},\;
|\langle v,e^{(m)}_i\rangle|\geq \frac{\delta}{n}\right.\\
&\left.\mbox{ \ \ \ \ \ \ \ \ \ \ and }
|\langle v,e^{(m)}_j\rangle|<\frac{\delta}{n}\mbox{ for }j\geq k_{\alpha+1}\right\},
\end{align*}
and let
$$
B^{(m)}_{\gamma+1}=\left\{v\in S^{n-1}:\,\exists i\geq k_{\gamma+1},\;
|\langle v,e^{(m)}_i\rangle|\geq \frac{\delta}{n}\right\}.
$$
This satisfies by
\eqref{aieninE},
 that if $\alpha=1,\ldots,\gamma+1$ and $v\in B^{(m)}_{\alpha}$, then
\begin{equation}
\label{hCmBalpha2}
h_{C_m}(v)\geq \frac{\delta}{n}\cdot a^{(m)}_{k_\alpha}.
\end{equation}
For $\alpha=1,\ldots,\gamma$, we consider the $G$-invariant subspace
$$
L^{(m)}_\alpha=\sum_{\eta=1}^\alpha\widetilde{L}^{(m)}_\eta
\mbox{ \ with }{\rm dim}\,L^{(m)}_\alpha=\sum_{\eta=1}^\alpha d_\eta,
$$
and hence if $\eta\leq \alpha$, then
\begin{equation}
\label{BjPsiLi2}
B^{(m)}_\eta\subset \Psi(L^{(m)}_\alpha\cap S^{n-1},\delta).
\end{equation}
Since $\lim_{m\to\infty}\frac{a^{(m)}_{q+1}}{a^{(m)}_q}=\infty$, if $m$ is large, then \eqref{tildeVqCmpupG} yields that
\begin{equation}
\label{tildeVqCmpupG2}
\frac{\aleph}2\leq
\left(\log \frac{a^{(m)}_{q+1}}{a^{(m)}_q}\right)
\prod_{\alpha=1}^{\gamma}(a^{(m)}_{k_\alpha})^{d_\alpha}.
\end{equation}

It follows that $S^{n-1}$ is partitioned into the Borel sets $B^{(m)}_1,\ldots,B^{(m)}_{\gamma+1}$, and
\eqref{BjPsiLi2} yields that if $\alpha=1,\ldots,\gamma$, then
\begin{eqnarray}
\label{muAi2}
\mu(B^{(m)}_1)+\ldots+\mu(B^{(m)}_\alpha)&\leq &\frac{(1-\tau)\sum_{\eta=1}^\alpha d_\eta}{q},\\
\label{muAn2}
\mu(B^{(m)}_1)+\ldots+\mu(B^{(m)}_{\gamma+1})&=&1.
\end{eqnarray}
For $\zeta=\frac{1-\tau}{q}$, we have $0< \zeta<\frac1{q}$, and we define
\begin{eqnarray}
\label{betai2}
\beta_\alpha&= &\mu(B_\alpha)-d_\alpha\zeta\mbox{ \ for $\alpha=1,\ldots,\gamma$}\\
\label{betan2}
\beta_{\gamma+1}&=&\mu(B_{\gamma+1})-\tau,
\end{eqnarray}
where (\ref{muAi2}) and (\ref{muAn2}) yield
\begin{eqnarray}
\label{sumbetai2}
\beta_1+\ldots+\beta_\alpha&\leq &0\mbox{ \ for $\alpha=1,\ldots,\gamma$}\\
\label{sumbetan2}
\beta_1+\ldots+\beta_{\gamma+1}&=&0.
\end{eqnarray}
We deduce from applying \eqref{hCmBalpha2}, \eqref{muAn2},  \eqref{betai2}, and \eqref{betan2}
that
\begin{align*}
\int_{S^{n-1}}\log h_{C_m}\,d\mu= &
\sum_{\alpha=1}^{\gamma+1}\int_{B^{(m)}_\alpha}\log h_{C_m}\,d\mu\\
\geq &\sum_{\alpha=1}^{\gamma+1}\mu(B^{(m)}_\alpha)\log a^{(m)}_{k_\alpha}+\sum_{\alpha=1}^{\gamma+1}\mu(B^{(m)}_\alpha)\log \frac{\delta}{n}\\
= & \sum_{\alpha=1}^{\gamma+1}\mu(B^{(m)}_\alpha)\log a^{(m)}_{k_\alpha}+\log \frac{\delta}{n}\\
=&
\sum_{\alpha=1}^{\gamma+1}\beta_\alpha\log a^{(m)}_{k_\alpha}+\sum_{\alpha=1}^{\gamma}d_\alpha\zeta \log a^{(m)}_{k_\alpha}+
\tau\log a^{(m)}_{k_{\gamma+1}}+\log \frac{\delta}{n}.
\end{align*}
In turn, it follows from
\eqref{tildeVqCmpupG2}, \eqref{sumbetai2}, \eqref{sumbetan2},
and the fact that $a^{(m)}_{k_\alpha}\leq a^{(m)}_{k_{\alpha+1}}$ that
\begin{align*}
\int_{S^{n-1}}\log h_{C_m}\,d\mu\geq &
\sum_{\alpha=1}^{\gamma+1}\beta_\alpha\log a^{(m)}_{k_\alpha}+
\zeta\log \frac{\aleph}2-\zeta\log\log \frac{a^{(m)}_{q+1}}{a^{(m)}_q}
+\tau\log a^{(m)}_{k_\gamma}+\log \frac{\delta}{n}\\
= &
(\beta_1+\ldots+\beta_{\gamma+1})\log a^{(m)}_{k_{\gamma+1}}+\\
&+\sum_{\alpha=1}^{\gamma}(\beta_1+\ldots+\beta_\alpha)\left(\log a^{(m)}_{k_\alpha}-\log a^{(m)}_{k_{\alpha+1}}\right)+\\
& +\zeta\log \frac{\aleph}2-\zeta\log\log \frac{a^{(m)}_{q+1}}{a^{(m)}_q}
+\tau\log a^{(m)}_{k_{\gamma+1}}+\log \frac{\delta}{n}\\
\geq &\zeta\log \frac{\aleph}2-\zeta\log\log \frac{a^{(m)}_{q+1}}{a^{(m)}_q}
+\tau\log a^{(m)}_{k_{\gamma+1}}+\log \frac{\delta}{n}.
\end{align*}

Let $R_m=\log\frac{a^{(m)}_{q+1}}{a^{(m)}_q}$, and notice that $\lim_{m\to\infty}R_m=\infty$.
We deduce from \eqref{tildeVqCmpupG2} and $a^{(m)}_{k_\alpha}\leq a^{(m)}_{k_{\gamma}}$ for
$\alpha=1,\ldots,\gamma$, that
$a^{(m)}_{k_{\gamma}}\geq \aleph_0 R_m^{\frac{-1}q}$  for $\aleph_0=(\aleph/2)^{\frac1q}$, and hence
$$
\log a^{(m)}_{k_{\gamma+1}}\geq \log \left(\aleph_0 R_m^{\frac{-1}q}e^{R_m}\right)=
\log \aleph_0-\frac1q\log R_m+R_m.
$$
Therefore, there exist $\aleph_1\in\R$ and $\aleph_2>0$ independent of $m$ such that
$$
\int_{S^{n-1}}\log h_{C_m}\,d\mu\geq \aleph_1-\aleph_2\log R_m+\tau R_m,
$$
and hence $\lim_{m\to\infty}R_m=\infty$ yields that
$\lim_{m\to\infty}\int_{S^{n-1}}\log h_{C_m}\,d\mu=\infty$, completing the proof of Proposition~\ref{q0-G-weighted-sufficient-limit}.
\proofbox

We will use the following property about invariant measures.

\begin{lemma}
\label{Invariant-Measure}
Let $n\geq 2$, and let $G\subset O(n)$ be a closed subgroup. If $\mu_1$ and $\mu_2$ are $G$-invariant finite Borel measures on $S^{n-1}$, then $\mu_1=\mu_2$ if and only if $\int_{S^{n-1}}\varphi\,d\mu_1=\int_{S^{n-1}}\varphi\,d\mu_2$ for any $G$-invariant continuous function $\varphi:S^{n-1}\to\R$.
\end{lemma}
\proof It is equivalent to prove that
\begin{equation}
\label{invariant-measure-continuous}
\int_{S^{n-1}}\psi\,d\mu_1=\int_{S^{n-1}}\psi\,d\mu_2
\end{equation}
 for any continuous function $\psi:S^{n-1}\to\R$ (without assuming $G$-invariance). Since $G$ is a compact group, there exists a $G$-invariant probability measure (Haar measure) $\nu$ on $G$. For any $u\in S^{n-1}$, let $\varphi(u)=\int_G\psi(gu)\,d\nu(g)$. It follows that $\varphi$ is a $G$-invariant continuous function, and Fubini's theorem and the $G$-invariance of $\mu_i$ yield that
$\int_{S^{n-1}}\varphi\,d\mu_i=\int_{S^{n-1}}\psi\,d\mu_i$ for $i=1,2$. In turn, we conclude \eqref{invariant-measure-continuous}.
\proofbox

\noindent{\bf Proof of Theorem~\ref{q0-G-weighted-sufficient0}}
First, we prove that there exists $\widetilde{K}\in\mathcal{C}$ such that
\begin{equation}
\label{tildeCminimizes}
\mathcal{E}(\widetilde{K})=\min\{\mathcal{E}(C):\,C\in\mathcal{C}\}.
\end{equation}
Let $C_m\in\mathcal{C}$ be such that
$$
\lim_{m\to\infty}\mathcal{E}(C_m)=\inf\{\mathcal{E}(C):\,C\in\mathcal{C}\}.
$$
It follows from  Proposition~\ref{q0-G-weighted-sufficient-limit}
 that there exists some $R>1$ such that
\begin{equation}
\label{CminRBn}
C_m\subset RB^n\mbox{ \ for each $C_m$}.
\end{equation}
Since $\widetilde{V}_q(C_m,Q)=1$, we deduce from  Lemma~\ref{inradiusVq} and \eqref{CminRBn} the existence of a $\xi>0$ such that
\begin{equation}
\label{xiBninCm}
\xi\,B^n\subset C_m\mbox{ \ for each $C_m$}.
\end{equation}
Combining the Blaschke Selection Theorem with \eqref{CminRBn} and \eqref{xiBninCm} yields that there exists a subsequence $\{C_{m'}\}\subset\{C_m\}$ tending to a centered convex body $\widetilde{K}$. We conclude from
\eqref{entropy-continuous} and the continuity of the dual intrinsic volume (cf. Lemma~\ref{CpqcontQ}), that
$\widetilde{K}\in\mathcal{C}$, and hence it satisfies \eqref{tildeCminimizes}.

Now the crucial claim is
\begin{equation}
\label{Step3-claim}
\mu=\widetilde{C}_q(\widetilde{K},Q;\cdot).
\end{equation}
According to Lemma~\ref{Invariant-Measure}, \eqref{Step3-claim} is equivalent with saying that
\begin{equation}
\label{Step3-claim0}
\int_{S^{n-1}}\frac{\varphi}{h_{\widetilde{K}}}\,d\mu=
 \int_{S^{n-1}}\frac{\varphi}{h_{\widetilde{K}}}\,d\widetilde{C}_q(\widetilde{K},Q;\cdot)
\end{equation}
for any continuous $G$-invariant function $\varphi:\,S^{n-1}\to\R$.
For $t\geq 0$, we consider the Wulff-shape
$$
K_t=\{x\in\R^n:\langle x,u\rangle\leq h_{\widetilde{K}}(u)+t\varphi(u)\;\forall u\in S^{n-1}\}.
$$
Hence $K_0=\widetilde{K}$, and the variational formula \eqref{dualAlexandrov} yields that
\begin{equation}
\label{Step3-problem-K-vol}
\left.\frac{d }{dt}\,\widetilde{V}_q(K_t,Q)\right|_{t=0}=q\int_{S^{n-1}}\frac{\varphi}{h_{\widetilde{K}}}\,d \widetilde{C}_q(\widetilde{K},Q;\cdot).
\end{equation}
If $|t|$ is small, then we consider the differentiable function
$$
f(t)=\int_{S^{n-1}}\log(h_{\widetilde{K}}+t \varphi)\,d\mu-\frac1q\log \widetilde{V}_q(K_t,Q)
$$
which, by the fact that
$\widetilde{V}_q(K_t,Q)^{\frac{-1}q}\cdot K_t\in\mathcal{C}$,
 $h_{K_t}\leq h_K(u)+t\varphi(u)$,
\eqref {entropy-rescale-invariant},
 and  \eqref{tildeCminimizes}, satisfies
 that
$$
f(t)\geq
\mathcal{E}(K_t)
=\mathcal{E}\left(\widetilde{V}_q(K_t,Q)^{\frac{-1}q}\cdot K_t\right)
\geq f(0).
$$
In particular, $f$ has a minimum at $t=0$, and hence \eqref{Step3-problem-K-vol} and
$\widetilde{V}_q(K_0,Q)=1$ imply that
$$
0=f'(0)=\int_{S^{n-1}} \frac{\varphi}{h_{\widetilde{K}}}\,d\mu-\int_{S^{n-1}}\frac{\varphi}{h_{\widetilde{K}}}\,d \widetilde{C}_q(\widetilde{K},Q;\cdot).
$$
This proves \eqref{Step3-claim0} and in turn \eqref{Step3-claim}.
\proofbox

\noindent{\bf Proof of Theorem~\ref{q01-G-weighted}.} Since a dual curvature measure is never concentrated on a great subsphere, Theorem~\ref{q0-G-weighted-sufficient0} yields that this property characterizes a $G$-invariant dual curvature measure for $q\in(0,1]$.
\proofbox

\section{Proof of Theorem~\ref{q0n-G-weighted}}

%We call two compact convex sets $K,C\subset\R^d$ homothetic if there exists $\gamma>0$ and $z\in\R^d$ such that
%$K=\gamma\,C+z$.
By induction on the number $k$ of summands, the Brunn-Minkowski inequality (cf. Schneider \cite{Sch14}) yields that
if $M_1,\ldots,M_k\subset\R^d$ for $d\geq 1$ are bounded convex sets, and $\lambda_1,\ldots,\lambda_k>0$, then
\begin{equation}
\label{BMmulti}
\HH^d\left(\sum_{i=1}^k\lambda_iM_i\right)^{\frac1d}\geq \sum_{i=1}^k\lambda_i \HH^d\left(M_i\right)^{\frac1d}.
\end{equation}
%with equality if and only if $M_1,\ldots,M_k$ are pairwise homothetic.
For a compact convex set $C\subset \R^n$, we write $\sigma(C)$ to denote the centroid with respect to the affine hull of $C$. We note that $\sigma(C)\in{\rm relint}\,C$, and $\sigma(\Phi C)=\Phi\sigma(C)$ for any $\Phi\in{\rm GL}(n)$ (cf. Schneider \cite{Sch14}). In turn, we deduce  the following property: for a closed subgroup $G\subset O(n)$ without a non-zero fixed point, if $X\subset \R^n$ is compact, then
\begin{equation}
\label{centroid-G}
\sigma(C)=o\mbox{ \ for \ }C={\rm conv}\{gX:\,g\in G\}.
\end{equation}
We also use the fact that if $f:\,\R^d\to[0,\infty]$ for $d\geq 1$ is a measurable function such that $f(x)<\infty$ for $x\neq o$, then
\begin{equation}
\label{levelsets-integral}
\int_{\R^d}f(x)\,dx=\int_0^\infty\HH^d(\{f>\alpha\})\,d\alpha,
\end{equation}
where if one of the two integrals is infinite, then the same holds for the other.
To simplify notation, if $M\subset \R^d$ is a $d$-dimensional convex set, then write
$$
\int_{M}f(x)\,d\HH^d(x)=\int_{M}f(x)\,dx.
$$

\begin{lemma}
\label{section}
Let $\Pi\subset \R^n$ $n\geq 2$ be a $G$-invariant proper linear subspace with $1\leq d={\rm dim}\Pi<n$ for a closed subgroup $G\subset O(n)$ without non-zero fixed point. Let also $f:\R^n\backslash \to[0,\infty]$ be a  $G$-invariant  function such that   $f(x)<\infty$ for $x\neq o$, and assume $\{f>\alpha\}$ is bounded and convex  for any $\alpha\in(0,f(o))$. If $M\subset z+\Pi$ is a $d$-dimensional compact convex set for a $z\in \Pi^\bot\backslash\{o\}$,
and $K={\rm conv}\{gM:g\in G\}$, then for $\lambda\in[0,1]$, we have
\begin{equation}
\label{section-eq}
\int_{K\cap (\lambda z+\Pi^\bot)}f\HH^d\geq\int_{M}f\HH^d.
\end{equation}
If, in addition, $o\in{\rm int}\{f>\alpha\}$ for any $\alpha\in(0,f(o))$, and
for any $\eta>0$, there exists $\alpha\in(0,f(o))$ such that $\{f>\alpha\}\subset \eta\,B^n$, then
there exists $\lambda_0\in(0,1)$ depending on $G$ and $M$ such that the
strict inequality holds in \eqref{section-eq} for $\lambda\in[0,\lambda_0)$.
\end{lemma}
\proof We set $\HH^m(\emptyset)=0$ for any $m>0$. We may assume that $\lambda<1$ (and hence $\lambda\in[0,1)$), and since ${\rm lin}\,K$ is invariant under $G$, we may also  assume that $K$ is $n$-dimensional. In addition, we may also assume that
$$
\int_{M}f\HH^d>0, \mbox{ \ and hence }
A=\sup\{f(x):\,x\in M\}>0,
$$
where   $f(o)=\sup\{f(x):\,x\in \R^n\}\geq A$ since $f$ is $G$-invariant, and the level sets are convex.

Now $\Pi^\bot$ is also a $G$-invariant linear subspace, and $C={\rm conv}\{gz:g\in G\}\subset \Pi^\bot$ is a $G$-invariant compact convex set. Since $\sigma(C)=o$, we deduce the existence of $\tilde{g}_1,\ldots,\tilde{g}_k\in G$, $k\geq 2$, and $\tilde{\lambda}_1,\ldots,\tilde{\lambda}_k>0$ with $\sum_{i=1}^k\tilde{\lambda}_i=1$ such that
$o=\sum_{i=1}^k\tilde{\lambda}_i\tilde{g}_iz$. For $g_i=\tilde{g}_1^{-1}\tilde{g}_i\in G$, $i=1,\ldots,k$, we have $g_1={\rm Id}$, and $o=\sum_{i=1}^k\tilde{\lambda}_ig_iz$. Therefore,
\begin{equation}
\label{lambdazgiz}
\lambda z=\lambda z +(1-\lambda)\left(\sum_{i=1}^k\tilde{\lambda}_ig_iz\right)=\sum_{i=1}^k\lambda_ig_iz
\end{equation}
where $\lambda_1,\ldots,\lambda_k>0$,  $\sum_{i=1}^k\lambda_i=1$, and $g_1z=z$. We set
$$
M_i=g_iM\mbox{ \ for \ }i=1,\ldots,k\mbox{ \ and \ }M_\lambda=K\cap (\lambda z+\Pi^\bot)
$$
where $M_1=M$.
It follows from \eqref{lambdazgiz} that $\sum_{i=1}^k\lambda_iM_i\subset M_\lambda$ and from the $G$-invariance of $f$ that $A=\sup\{f(x):\,x\in M_i\}$ for $i=1,\ldots,k$. Since the level sets of $f$ are convex, if $\alpha\in(0,A)$, then
$$
\sum_{i=1}^k\lambda_i(M_i\cap \cap \{f>\alpha\})\subset M_\lambda\cap \{f>\alpha\},
$$
and hence the Brunn-Minkowski inequality \eqref{BMmulti} and the $G$-invariance of $f$ yields that
\begin{equation}
\label{Mlambda-levelset-est}
\HH^d\left(M_\lambda\cap \{f>\alpha\}\right)\geq \HH^d\left(M\cap \{f>\alpha\}\right).
\end{equation}
We conclude \eqref{section-eq} since \eqref{levelsets-integral} implies that
\begin{align}
\nonumber
\int_{M_\lambda}f(x)\,dx&=\int_0^\infty\HH^d\left(M_\lambda\cap \{f>\alpha\}\right)\,d\alpha
\geq \int_0^A\HH^d\left(M_\lambda\cap \{f>\alpha\}\right)\,d\alpha\\
\label{Mlambda-int-levelset-est}
&\geq \int_0^A\HH^d\left(M\cap \{f>\alpha\}\right)\,d\alpha=
\int_{M}f(x)\,dx.
\end{align}

Now we assume that  $o\in{\rm int}\{f>\alpha\}$ for any $\alpha\in(0,f(o))$, and
for any $\eta>0$, there exists $A_\eta\in(0,f(o))$ such that $\{f>A_\eta\}\subset \eta\,B^n$. We choose $\eta>0$ such that $M\cap \eta B^n=\emptyset$, and hence $A_\eta\geq A$. Next,  we choose an $\widetilde{A}_\eta\in (A_\eta,f(o))$ and $r_\eta>0$ such that $r_\eta B^n\subset \{f>\widetilde{A}_\eta\}$.

We deduce from \eqref{centroid-G} that $\sigma(K)=o$. It follows that $M_0=K\cap \Pi$ is a $d$-dimensional compact convex set and is $G$-invariant, and hence $\sigma(M_0)=o$ according to  \eqref{centroid-G}. We deduce that there exists
$\lambda_0>0$ such that
$$
\HH^d(M_\lambda\cap r_\eta B^n)>0 \mbox{ \ if }0\leq \lambda<\lambda_0.
$$
Thus, if $0\leq \lambda<\lambda_0$, then
$$
\int_{A_\eta}^{\widetilde{A}_\eta}\HH^d\left(M_\lambda\cap \{f>\alpha\}\right)\,d\alpha\geq
\HH^d\left(M_\lambda\cap \{f>\widetilde{A}_\eta\}\right)\cdot(\widetilde{A}_\eta-A_\eta)>0.
$$
It follows from the fact that $A_\eta\geq A$ and \eqref{Mlambda-int-levelset-est},  that
\begin{align*}
\int_{M_\lambda}f(x)\,dx&=\int_0^\infty\HH^d\left(M_\lambda\cap \{f>\alpha\}\right)\,d\alpha
> \int_0^A\HH^d\left(M_\lambda\cap \{f>\alpha\}\right)\,d\alpha\\
&\geq \int_0^A\HH^d\left(M\cap \{f>\alpha\}\right)\,d\alpha=
\int_{M}f(x)\,dx,
\end{align*}
verifying the strict inequality in \eqref{section-eq}.
\proofbox

\noindent{\bf Proof of Theorem~\ref{q0n-G-weighted}:}
We shall prove that if
$G\subset O(n)$ is a closed subgroup without a non-zero fixed point,
$K,Q\in\mathcal{K}^n_{o}$ are invariant under $G$,
and $L\subset\R^n$ is a $G$-invariant linear $m$-dimensional subspace with $m=1,\ldots,n-1$, then
\begin{equation}
\label{q0n-G-weighted-proof}
\widetilde{C}_q(K,Q;L\cap S^{n-1})< \frac{m}{q}\cdot \widetilde{V}_q(K,Q).
\end{equation}
Recall that \eqref{dualintrvolQ} yields
\begin{equation}
\label{dualintrvolQ0}
\widetilde{V}_q(K,Q)=\frac{q}n\int_K\|x\|_Q^{q-n}\,dx.
\end{equation}
On the other hand, writing
$$
\Xi=\bigcup\left\{{\rm conv}\{o,x\}:\,x\in {\pmb\nu}_K^{-1}(L\cap S^{n-1})\right\},
$$
 then \eqref{dual-curvature-integral} implies that
\begin{equation}
\label{dual-curvature-integral0}
\widetilde{C}_q(K,Q;L\cap S^{n-1})=\frac{q}n\int_{\Xi}\|x\|_Q^{q-n}\,dx.
\end{equation}

We write $\widetilde{K}$ to denote the ($G$-invariant) orthogonal projection of $K$ into $L$ and
$\partial \widetilde{K}$ to denote the relative boundary of $\widetilde{K}$. For $u\in L\cap S^{n-1}$, let
$\varrho_{\widetilde{K}}(u)>0$ be the radial function of $\widetilde{K}$; namely,
$\varrho_{\widetilde{K}}(u)\cdot u\in\partial \widetilde{K}$.  We define
$M_x=K\cap (x+L^\bot)$ for $x\in \widetilde{K}$, and if $x=\varrho_{\widetilde{K}}(u)\cdot u$ for $u\in L\cap S^{n-1}$, then
we set $\widetilde{M}_u=M_x$. Setting $d=n-m={\rm dim}\,L^\bot$ and using Fubini's theorem and polar coordinates in $L$, \eqref{dualintrvolQ0} yields
$$
\widetilde{V}_q(K,Q)=\frac{q}n\int_{L\cap S^{n-1}}\int_0^{\varrho_{\widetilde{K}}(u)}r^{m-1}
\int_{M_{ru}}\|y\|_Q^{q-n}\,d\HH^d(y)\,drdu.
$$
Here, Lemma~\ref{section} implies that for any $u\in L\cap S^{n-1}$, there exists $r_u>0$ such that
if $0\leq r\leq \varrho_{\widetilde{K}}(u)$, then
$$
\int_{M_{ru}}\|y\|_Q^{q-n}\,d\HH^d(y)\geq \int_{\widetilde{M}_u}\|y\|_Q^{q-n}\,d\HH^d(y).
$$
And if in addition $r<r_u$, then
$$
\int_{M_{ru}}\|y\|_Q^{q-n}\,d\HH^d(y)> \int_{\widetilde{M}_u}\|y\|_Q^{q-n}\,d\HH^d(y).
$$
Therefore,
\begin{align}
\nonumber
\widetilde{V}_q(K,Q)&>\frac{q}n\int_{L\cap S^{n-1}}\int_0^{\varrho_{\widetilde{K}}(u)}r^{m-1}
\int_{\widetilde{M}_u}\|y\|_Q^{q-n}\,d\HH^d(y)\,drdu\\
\label{dualintrvolQ-est}
&=\frac{q}{nm}\int_{L\cap S^{n-1}}\varrho_{\widetilde{K}}(u)^m
\int_{\widetilde{M}_u}\|y\|_Q^{q-n}\,d\HH^d(y)\,du.
\end{align}

Concerning the dual curvature measure, if $u\in L\cap S^{n-1}$ and $0< r< \varrho_{\widetilde{K}}(u)$, then
the $\Xi$ in \eqref{dual-curvature-integral0} satisfies (recall that $n=d+m$) that
$$
\int_{\Xi\cap (ru+L^\bot)}\|y\|_Q^{q-n}\,d\HH^d(y)=
\frac{r^{q-m}}{\varrho_{\widetilde{K}}(u)^{q-m}}\int_{\widetilde{M}_u}\|y\|_Q^{q-n}\,d\HH^d(y).
$$

Therefore, using Fubini's theorem and converting to polar coordinates in $L$ to evaluate \eqref{dual-curvature-integral0}, we deduce  that
\begin{align*}
\widetilde{C}_q(K,Q;L\cap S^{n-1})&=\frac{q}n\int_{L\cap S^{n-1}}\int_0^{\varrho_{\widetilde{K}}(u)}\frac{r^{q-1}}{\varrho_{\widetilde{K}}(u)^{q-m}}
\int_{M_{ru}}\|y\|_Q^{q-n}\,d\HH^d(y)\,drdu\\
&=\frac{1}{n}\int_{L\cap S^{n-1}}\varrho_{\widetilde{K}}(u)^m
\int_{\widetilde{M}_u}\|y\|_Q^{q-n}\,d\HH^d(y)\,du.
\end{align*}
Comparing this formula to \eqref{dualintrvolQ-est} yields \eqref{q0n-G-weighted-proof}.
\proofbox

\section{Proof of Theorem~\ref{qn+1-Lpzonoid-weighted} }

We recall that for $p\geq 1$ and a finite even Borel measure $\mu$ on $S^{n-1}$ not concentrated on any great subsphere, the polar $L_p$ zonoid $Z_p^*(\mu)$ is the unit ball of the norm
\begin{equation}
\label{polarLpzondoid-latedef}
\|x\|_{Z_p^*(\mu)}=\left(\int_{S^{n-1}}|\langle x,u\rangle|^p\,d\mu(u)\right)^{\frac1p},\qquad x\in\R^n.
\end{equation}
We recall Theorem~\ref{qn+1-Lpzonoid-weighted} in the following form.

\begin{theo}
\label{qn+1-Lpzonoid-weighted0}
Let $n\geq 2$ and $q\geq n+1$. If $K,Q\in\mathcal{K}^n_{o}$ where
Q is a polar $L_{q-n}$ zonoid, then
\begin{equation}
\label{subspace-concentration-G0}
\widetilde{C}_{q}(K,Q;L\cap S^{n-1})<\mbox{$\frac{q-n+{\rm dim}\,L}q$}\cdot \widetilde{C}_{q}(K,Q;S^{n-1})
\end{equation}
for any proper linear subspace $L\subset \R^n$.
\end{theo}

As the core statement within the proof of Theorem~1.9  in Henk-Pollehn \cite{HeP18},
the paper \cite{HeP18} proves the following result as equation (3.4).

\begin{lemma}
\label{Henk-Pollehn}
Let $p\geq 1$, $u\in S^{n-1}$, $d\in\{1,\ldots,n-1\}$ and $K_0,K_1\subset \R^n$ be $d$-dimensional compact convex sets whose affine hulls are parallel. If $K_\lambda=(1-\lambda)K_0+\lambda K_1$ for $\lambda\in[0,1]$, then  for any $\lambda\in[0,1]$, we have
\begin{align}
\label{Henk-Pollehn-eq}
&\int_{K_\lambda}|\langle x,u\rangle|^p\,\HH^d(x)+\int_{K_{1-\lambda}}|\langle x,u\rangle|^p\,\HH^d(x)\\
\nonumber
\geq & |2\lambda-1|^p\left(\int_{K_0}|\langle x,u\rangle|^p\,\HH^d(x)+\int_{K_1}|\langle x,u\rangle|^p\,\HH^d(x)\right).
\end{align}
\end{lemma}

\begin{coro}
\label{Henk-Pollehn-cor}
If $p\geq 1$, $\tau\in[0,1]$,  $z\in \R^n\backslash \{o\}$, $M\subset z+z^\bot$ is an $d$-dimensional compact convex set for $d\in\{1,\ldots,n-1\}$, and $Q=Z_p^*(\mu)$ for a finite even Borel measure $\mu$ on $S^{n-1}$ not concentrated on any great subsphere, then
\begin{equation}
\label{Henk-Pollehn-cor-eq}
\int_{(\tau z+z^\bot)\cap {\rm conv}\,\{M,-M\}}\|x\|_Q^p\,d\HH^{d}(x)\geq \tau^p\int_{M}\|x\|_Q^p\,d\HH^{d}(x).
\end{equation}
\end{coro}
\begin{proof}
We observe that $\tau z=(1-\lambda)z+\lambda (-z)$ for $\lambda=\frac{1-\tau}2\in[0,1]$ and $\tau=|2\lambda-1|$, and hence
\begin{equation}
\label{tau-lambda}
M_\lambda:=(1-\lambda)M+\lambda (-M)=(\tau z+z^\bot)\cap {\rm conv}\,\{M,-M\}.
\end{equation}

We deduce from \eqref{polarLpzondoid-latedef}, and then from $M_{1-\lambda}=-M_\lambda$, \eqref{Henk-Pollehn-eq}, and \eqref{tau-lambda} that
\begin{align*}
\int_{M_\lambda}\|x\|_Q^p\,d\HH^{d}(x)=&
\int_{S^{n-1}}\int_{M_\lambda}|\langle x,u\rangle|^p\,d\HH^{d}(x)\,d\mu(u)\\
\geq &\tau^p\int_{S^{n-1}}\int_{M}|\langle x,u\rangle|^p\,d\HH^{d}(x)\,d\mu(u)=\tau^p\int_{M}\|x\|_Q^p\,d\HH^{d}(x).
\end{align*}
\end{proof}

\begin{proof}[Proof of Theorem~\ref{qn+1-Lpzonoid-weighted0}]
For any $z\in K|L$  and $u\in S^{n-1}\cap L$, we write
\begin{align*}
M_z=&(z+L^\bot)\cap K,\\
\widetilde{M}_u=&(\varrho_{K|L}(u)u+L^\bot)\cap K.
\end{align*}

In order to shorten formulas, we write $dy$ instead of $d\HH^k(y)$ where value of $k$ will be clear from the context.
We deduce from Corollary~\ref{Henk-Pollehn-cor} and $q\geq n+1$ that
\begin{align*}
\widetilde{C}_q(K,Q,S^{n-1})=&\frac{q}{n}\int_K\|x\|_Q^{q-n}\,dx=
\frac{q}{n}\int_{K|L}\int_{M_z}\|x\|_Q^{q-n}\,dx\,dz\\
\geq &\frac{q}{n}\int_{L\cap S^{n-1}}r^{m-1}\int_0^{\varrho_{K|L}(u)}\int_{(ru+L^\bot)\cap \Gamma_u}
\| x\|_Q^{q-n}\,dx \,dr du\\
\geq &\frac{q}{n}\int_{L\cap S^{n-1}}\int_0^{\varrho_{K|L}(u)}r^{m-1}\cdot \left(\frac{r}{\varrho_{K|L}(u)}\right)^{q-n}\int_{\widetilde{M}_u}\|x\|_Q^{q-n}\,dx\,dr \,du\\
=&\frac{q}{n(q-n+m)}\int_{L\cap S^{n-1}}\varrho_{K|L}(u)^{m}\int_{\widetilde{M}_u}\|x\|_Q^{q-n}\,dx\,dr \,du
\end{align*}

In order to determine $\widetilde{C}_q(K,Q;L\cap S^{n-1})$, we observe that
if $u\in L\cap S^{n-1}$ and $0< r< \varrho_{K|L}(u)$, then
the $\Xi$ in \eqref{dual-curvature-integral} satisfies that
$$
\int_{\Xi\cap (ru+L^\bot)}\|y\|_Q^{q-n}\,d\HH^d(y)=
\frac{r^{q-m}}{\varrho_{K|L}(u)^{q-m}}\int_{\widetilde{M}_u}\|y\|_Q^{q-n}\,d\HH^d(y).
$$
It follows from \eqref{dual-curvature-integral}, the use of polar coordinates in $L$ and Fubini's theorem that
$$
\widetilde{C}_q(K,Q;L\cap S^{n-1})=\frac{1}{n}\int_{L\cap S^{n-1}}\varrho_{K|L}(u)^m
\int_{\widetilde{M}_u}\|y\|_Q^{q-n}\,d\HH^d(y)\,du,
$$
proving Theorem~\ref{qn+1-Lpzonoid-weighted0}.
\end{proof}

\end{document}